\documentclass[final,leqno,letterpaper]{amsart}

\usepackage{amsmath,amssymb,amsfonts,mathrsfs,datetime,mathtools,stmaryrd,mathtools,bm,chngcntr,diagbox,xcolor,placeins,subcaption}
\usepackage[toc,page]{appendix}
\usepackage[space,noadjust,nocompress]{cite}
\usepackage{todonotes}

\usepackage[mathlines]{lineno}
\makeatletter                                                                                                                                                            

\newtheorem{assumption}{Assumption}[section]
\newtheorem{thm}{Theorem}[section]

\newtheorem{lem}{Lemma}[section]

\counterwithin{table}{subsection}
\counterwithin{figure}{subsection}

\usepackage{cancel}

\newcommand{\changed}[1]{{\color{black}#1}}
\def\mesh{\mathbb T}
\def\cell{\kappa}
\def\refcell{\widehat{\cell}}
\def\ip{\delta_{\text{IP}}}
\newcommand{\gi}{g}
\newcommand{\gj}{{g'}}

\title[Multilevel Schwarz Reaction-Diffusion]{Multilevel Schwarz preconditioners for singularly perturbed symmetric reaction-diffusion systems}

\author{Jos\'e Pablo Lucero Lorca \and Guido Kanschat}

\thanks{Universit\'e de Gen\`eve, S\'ection de Math\'ematiques ({\tt pablo.lucero@unige.ch})}

\thanks{Interdisciplinary Center for Scientific Computing (IWR), Heidelberg University,({\tt kanschat@uni-heidelberg.de}), supported in part by DFG through priority program 1648 SPPEXA and by the special program ``Numerical Analysis of Complex PDE Models in the Sciences'' of the Erwin Schrödinger International Institute for Mathematics and Physics (ESI) in Vienna. Computations were run in part on the bwForDev cluster at Heidelberg University.}
\begin{document}
\maketitle
\begin{abstract}
  We present robust and highly parallel multilevel non-overlapping Schwarz preconditioners, to solve an interior penalty discontinuous Galerkin finite element discretization of a system of steady state, singularly perturbed reaction-diffusion equations with a singular reaction operator, using a GMRES solver. We provide proofs of convergence for the two-level setting and the multigrid V-cycle as well as numerical results for a wide range of regimes.


  \smallskip
  \noindent
  65N55, 65N30, 65J10, 65F08
\end{abstract}

\section{Introduction}
In this paper, we present analysis and numerical experiments of two-level Schwarz preconditioners and their multilevel versions for an interior penalty discontinuous Galerkin finite element discretization of a system of reaction-diffusion equations \changed{not requiring special mesh structures resolving boundary layers.} Our focus is on the singularly perturbed case, where the reaction system has an inertial subspace. \changed{We use a massively parallel smoother as in \cite{Dryja2016}, therefore we provide new convergence estimates for elliptic and reaction-diffusion systems including quadrilateral and hexahedral meshes.}
\changed{The} estimates are robust with respect to the parameters of the system and the experiments confirm the efficiency of the method. 



Reaction-diffusion systems arise in a variety of physical, chemical and biological contexts. \changed{Due to conservation of mass, these systems are all characterized by an inertial subspace (an inertial manifold in the nonlinear case) on which the system reduces to  an almost reaction free diffusion equation. Nevertheless, the contributions orthogonal to this subspace are still important in applications and often cannot be neglected. Thus, numerical methods have to deal with long-ranged coupling in the inertial subspace as well as short-ranged behavior in its complement in an efficient way.}
One particular area where these models are widely used is radiation transport, where the reaction-diffusion equation is an approximation of Boltzmann's linear transport equation that becomes relevant in the so called \emph{diffusive} regimes, which are characterized by small mean free paths compared to the size of the domain. In these regimes the transport equation is nearly singular and its solution in the interior of the computational domain is close to the solution of a reaction-diffusion equation \cite{Manteuffel1998}.

We employ the interior penalty discontinuous Galerkin (IP-DG) method to discretize the singularly perturbed reaction-diffusion system in steady state. IP-DG \cite{Arnold2002,Nitsche1971,Baker1977,Arnold1982,Wheeler1978} methods are particularly interesting to solve reaction-diffusion equations since \changed{oscillations at boundary and interior layers (Gibbs phenomenon) are much less notable than with standard conforming finite elements for singularly perturbed problems \cite{LewBuscaglia2008}}. \changed{Thus, they produce better approximations if such layers are not resolved.} Using this discretization the reaction operator involves only volume integrals with no coupling between cells. Therefore, \changed{we expect that IP-DG is particularly well suited for Schwarz methods since contributions of the reaction term are included inside the local solvers.}

We solve the discrete problem with a GMRES solver using multilevel preconditioners with nonoverlapping Schwarz smoothers, effectively solving a full reaction-diffusion problem in each cell (see \cite{KanschatLucero16}). Convergence estimates for such methods applied to \changed{pure} diffusion problems have been developed in \cite{FengKarakashian2001}. There, it is assumed that the subdomains defining the decomposition of the fine space are unions of coarse cells. It is more efficient though to employ subdomains based on fine cells and recently, an analysis has been shown covering this case~\cite{Dryja2016}. However, its application does not extend to quadrilaterals and hexahedra, since the proof uses $P_1$ nonconforming interpolant and enriching operators for simplices, see~\cite{Brenner2003}. We provide an extension for quadrilaterals and hexahedra without such restrictions.

\changed{Subspace correction methods for singularly perturbed reaction-diffusion equations have been studied in~\cite{Boglaev1998,MacMullen01}. Both articles rely on strictly positive reaction terms and use Shishkin-type meshes for robust discretization of boundary layers. This technique is extended to singularly perturbed reaction-diffusion systems in~\cite{Stephens2009}. Also there, the authors assume a strictly positive definite reaction system, thus being able to make the assumption of an exponential boundary layer, but excluding the presence of an inertial subspace.}

\changed{As stated above, inertial subspaces can be important in applications. Since they do not exhibit boundary layers, Shishkin-type meshes will not be adapted to all solution components. We also do not want to necessarily have to resolve boundary layers, albeit not only the inertial part may be of importance. Thus, we do not solve the limit problem and propose a method which is robust in the sense that its iteration counts are uniformly bounded with respect to reaction parameters and mesh size.
}

Our main results are the proof of the stable decomposition shown in lemma \ref{lem:laplacianstabledecomposition} to obtain convergence estimates for two-level preconditioners, and the multigrid V-cycle preconditioner estimate in theorem \ref{thm:multigrid}.
The paper is structured as follows: in section \ref{sec:diffusionequation} we introduce the continuous problem and the IP-DG discretization. In section \ref{sec:solver} we develop two-level Schwarz and multigrid preconditioners and prove convergence estimates. Finally, in section \ref{sec:numericalexperiments} we demonstrate the efficiency of the proposed methods with experimental results.

\section{Model problem}\label{sec:diffusionequation}

We consider the following system of $G$ steady state reaction-diffusion equations with a singularly perturbed reaction term
\begin{align}\label{eqn:DiffusionEquation}
\begin{aligned}
-\nabla \cdot (\eta_\gi \nabla u_\gi) + \frac{1}{\varepsilon} \sum_{\gj=1}^G \changed{\sigma_{\gi \gj} \left(u_\gi - u_\gj  \right)} &= S_\gi && \text{ in } \Omega \text{, with } \gi=1\dots G,
\end{aligned}
\end{align}
where $\gi$ is the \emph{group} index identifying each reacting \emph{substance}, $\eta_\gi$ is the diffusion coefficient for each group $\gi$, $\varepsilon$ is a perturbation parameter defining the relative size of the reaction with respect to the diffusion term, $\Omega$ is a convex polyhedral domain in $\mathbb{R}^d$ with $d=2,3$ and $S_\gi$ is a known source. 

The equation is provided with the boundary conditions
\begin{align*}
\begin{aligned}
u_\gi = 0 && \text{ on } \Gamma \text{, with } \gi=1\dots G,
\end{aligned}
\end{align*}
where $\Gamma$ is the boundary of $\Omega$.

We assume $\eta_g,\sigma_{\gi\gj} \in C^{\infty}(\Omega)$ and $\sigma_{\gi\gj} \ge 0$, for all $\gi,\gj=1\dots G$ and there exists $C>0$ such that $\eta_\gi \ge C$ in $\Omega$. Furthermore, we assume that the reaction matrix is symmetric \changed{and singular} since
\begin{gather}
    \sigma_{gg} = -\sum_{g' \ne g} \sigma_{gg'}\qquad\forall g=1\dots G.
\end{gather}

We introduce the Hilbert spaces
\begin{align*}
\mathcal{V} = \left(H_0^1(\Omega)\right)^G &, & \mathcal{H} = \left(L^2(\Omega)\right)^G,
\end{align*}
where $H^1_0(\Omega)$ is the standard Sobolev space with zero boundary
trace. They are provided with inner products
\begin{align}\label{eqn:InnerProducts}
(u,v)_\mathcal{V} = \sum_{g=1}^G \left(\eta_\gi \nabla u_\gi,\nabla v_\gi \right)_{L^2(\Omega)} &, & (u,v)_{\mathcal{H}} = \sum_{\gi=1}^G \left(u_\gi,v_\gi \right)_{L^2(\Omega)},
\end{align} 
and norms
\begin{align*}
\|u\|_{\mathcal{V}}^2 = (u,u)_\mathcal{V} &, & \|u\|_{\mathcal{H}}^2 = (u,v)_{\mathcal{H}}.
\end{align*}
The weak form of problem~\eqref{eqn:DiffusionEquation} is: find $u \in \mathcal{V}$ such that
\begin{equation}
  \label{eqn:VariationalProblem}
  \mathscr{A}(u,v) = (f,v)_{\mathcal{H}},
\end{equation}
where $f \in \mathcal{H}$ and the bilinear form is
\begin{gather}\label{eqn:BilinearForm}
\begin{aligned}
\mathscr{A}(u,v) =& \sum_{\gi = 1}^G \int_\Omega \eta_\gi \nabla u_\gi \cdot \nabla v_\gi dx +  \frac{1}{\varepsilon} \sum_{\gi = 1}^G \sum_{\gj = 1}^G \int_\Omega \sigma_{\gi\gj} \left(u_\gi -  u_\gj \right) v_\gi dx \\
=& \left(\boldsymbol{D} \nabla u, \nabla v\right)_{\mathcal{H}} + \frac{1}{\varepsilon} \left(\boldsymbol{\Sigma} u, v\right)_{\mathcal{H}} = \left(u,v\right)_{\mathcal{V}} + \frac{1}{\varepsilon} \left(\boldsymbol{\Sigma} u, v\right)_{\mathcal{H}}.
\end{aligned}
\end{gather}
The second line uses the vector notation
\begin{align*}
u &= (u_1,\dots,u_G)^\intercal, & v &= (v_1,\dots,v_G)^\intercal,\\
\boldsymbol{D} &= \text{diag}(\eta_1,\dots,\eta_G), \text{ and } & \boldsymbol{\Sigma} =& 
\begin{psmallmatrix} 
\sigma_{11} & \dots & -\sigma_{G1} \\
\dots & \dots & \dots \\
-\sigma_{1G} & \dots & \sigma_{GG} \\
\end{psmallmatrix}.
\end{align*}

According to our assumptions, the reaction matrix $\boldsymbol{\Sigma}$ is a symmetric, weakly diagonally dominant singular M-matrix\footnote{We use the term \emph{singular} M-matrix, following the terminology in \cite[p. 119]{Horn1991}, to denote a matrix that can be expressed as $\boldsymbol{A} = s\boldsymbol{I} - \boldsymbol{B}$, where all the elements in $\boldsymbol{B}$ are nonnegative, $s$ is equal to the maximum of the moduli of the eigenvalues of $\boldsymbol{B}$, and $\boldsymbol{I}$ is an identity matrix.} with zero column and row sum. By the Perron-Frobenius theorem, this implies $\boldsymbol{\Sigma}$ is singular with rank less than $G$ and by the Ger\v{s}gorin circle theorem, all eigenvalues are nonnegative. 

Physically, the properties of $\boldsymbol{\Sigma}$ ensure substance conservation and the absence of sinks inside the domain. In a radiation transport context, this implies that the system can have no particle absorption and particles only disappear when they reach the boundary. The presence of absorption would imply all eigenvalues are positive and $\boldsymbol{\Sigma}$ would be invertible.

Under the assumptions on the parameters of equation \eqref{eqn:DiffusionEquation}, the bilinear form $\mathscr{A}(u,v)$ is continuous and $\mathcal{V}$-coercive relatively to $\mathcal{H}$ (see \cite[\S2.6]{DautrayLions1984}), i. e. there exist constants $\gamma_\mathscr{A},C_\mathscr{A} > 0$ such that
\begin{align*} 
\mathscr{A}(u,u) \ge \gamma_\mathscr{A} \|u\|_{\mathcal{V}}^2  &, & \mathscr{A}(u,v) \le C_\mathscr{A}\|u\|_{\mathcal{V}}\|v\|_{\mathcal{V}}.
\end{align*}
where we remark that even though $\gamma_\mathscr{A}$ is independent of $\varepsilon$, $C_\mathscr{A}$ is not. 

From Lax-Milgram's theorem, the variational problem admits a unique solution in $\mathcal{V}$.

\subsection{Discrete problem}\label{sec:discreteproblem}

We apply a IP-DG discretization to the bilinear form $\mathscr{A}(\cdot,\cdot)$ \cite{Arnold1982}. Let $\mesh_h$ be a subdivision of the domain $\Omega$ into quadrilaterals or hexahedra $\cell$, such that each cell $\cell$ is described by a $d$-linear mapping $\Phi_\cell$ from the reference cell $\refcell = [0,1]^d$ onto itself. Conformity of the faces of mesh cells is not required, but we assume local quasi-uniformity and shape regularity in the sense that the Jacobians of $\Phi_\cell$ and their inverses are uniformly bounded.

Let $\mathbb{Q}_p$ be the space of tensor product polynomials of degree up to $p$ in each coordinate direction. Then, define the mapped space $\mathbb{Q}_p(\cell)$ on the cell $\cell$ as the pull-back of functions under $\Phi_\cell$. The vector-valued, discontinuous function space $V_h$ is then defined as
\begin{gather}
V_h = \bigl\{ v\in \mathcal{H} \big| v_{|\cell} \in \mathbb{Q}_p^G(\cell) \bigr\}.
\end{gather}

Let $\mathbb{F}_h^I$ be the set of all interior faces of the mesh and $\mathbb{F}_h^B$ the set of all boundary faces. Let  $\cell_{+},\cell_{-} \in \mesh_h$ be two mesh cells with a joint face $F \in \mathbb{F}_h^I$, and let $u_{+}$ and $u_{-}$ be the traces of functions $u$ on $F$ from $\cell_{+}$ and $\cell_{-}$ respectively. On an \changed{interior} face $F$, we define the \emph{averaging} operator as
\begin{align*}
\{\!\!\{u\}\!\!\} = \frac{u_{+}+u_{-}}{\sqrt2}.
\end{align*}
\changed{On the boundary, there is only a single value and we set $\{\!\!\{u\}\!\!\} = u$.}

We introduce the following definition of mesh integrals
\begin{align*}
    \int_{\mesh_h} u \, dx = \sum_{\kappa \in \mesh_h} \int_\kappa u\, dx
\end{align*}
and integrals over $\mathbb{F}_h^I$ and $\mathbb{F}_h^B$ are defined accordingly.
The interior penalty (IP) bilinear form for the scalar Laplacian, as described in \cite{Arnold1982}, reads
\begin{multline}\label{eqn:DGIPLaplacian}
  \alpha_h(u,v) = \int_{\mesh_h} \nabla u \cdot \nabla v \, dx
  - \int_{\mathbb{F}_h^I \cup \mathbb{F}_h^B} \left( \{\!\!\{ u \boldsymbol{n} \}\!\!\} \cdot \{\!\!\{\nabla v \}\!\!\} + \{\!\!\{\nabla u \}\!\!\} \cdot \{\!\!\{ v \boldsymbol{n} \}\!\!\} \right) \,ds \\
+ \int_{\mathbb{F}_h^I \cup \mathbb{F}_h^B} \frac{\ip}{h} \{\!\!\{ u \boldsymbol{n} \}\!\!\} \cdot  \{\!\!\{ v \boldsymbol{n} \}\!\!\} \,ds
\end{multline}
where $h$ is the minimum cell diameter adjacent to the face, $u\boldsymbol{n} = \left(u_1 \boldsymbol{n},u_2 \boldsymbol{n},\ldots,u_G \boldsymbol{n}\right)^\intercal$ and $\nabla u = \left(\nabla u_1, \nabla u_2,\ldots,\nabla u_G \right)^\intercal$. We have replaced the jump operator used in \cite{Arnold1982} for the equivalent expression: $\sqrt2\{\!\!\{u \boldsymbol{n} \}\!\!\} = u_{+}\boldsymbol{n}_{+} + u_{-} \boldsymbol{n}_{-}$. Coercivity and continuity are proven in \cite{Arnold1982} under the assumption that $\ip$ is sufficiently large. We will assume in the following that this holds true.

We then define the discrete bilinear form, including the diffusion coefficients follows
\begin{multline}\label{eqn:DGIPoriginal}
  a_h(u,v) = \int_{\mesh_h} \boldsymbol{D} \nabla u \cdot \nabla v\, dx
  + \int_{\mathbb{F}_h^I \cup \mathbb{F}_h^B} 4 \frac{\ip}{h} \{\!\!\{ \boldsymbol{D} (u \boldsymbol{n})\}\!\!\} \cdot \{\!\!\{ v \boldsymbol{n} \}\!\!\} ds
  \\
  - \int_{\mathbb{F}_h^I \cup \mathbb{F}_h^B} 2 \left( \{\!\!\{ u \boldsymbol{n} \}\!\!\} \cdot \{\!\!\{\boldsymbol{D} \nabla v \}\!\!\} + \{\!\!\{\boldsymbol{D} \nabla u \}\!\!\} \cdot \{\!\!\{ v \boldsymbol{n} \}\!\!\} \right) \,ds
  .
\end{multline}
Under the assumptions made in the previous sections and $\ip$ sufficiently large, $a_h(u,v)$ is coercive and continuous.

Using~\eqref{eqn:DGIPoriginal}, our IP-DG formulation for the singularly perturbed reaction diffusion problem reads: Find $u \in V$ such that
\begin{gather}\label{eqn:IPDiffusion}
\mathcal{A}_h(u,v) \equiv a_h(u,v) + \frac{1}{\varepsilon} \int_{\mesh_h} \boldsymbol{\Sigma} u \cdot v dx = \int_{\mesh_h} S \cdot v dx \hspace{0.25cm} \forall v \in V_h,
\end{gather}
\changed{where $S$ is a right-hand side or \emph{source} term.}
We observe that given the non-negativeness of $\boldsymbol{\Sigma}$, the coercivity constant for our problem coincides with the Laplacian case while the continuity constant is now dependent on $\varepsilon$. In order to obtain a robust solver we precondition the problem to be able to bound the spectral radius of the preconditioned system independently of $\varepsilon$.

Finally, using a standard basis for the local finite element spaces on each cell and concatenating, we obtain the linear system
\begin{align*}
\boldsymbol{A} {\bf u} = {\bf f},
\end{align*}
where ${\bf u}$ and ${\bf f}$ are the coefficient vector of the representation of $u$ and $f$ respectively in terms of the chosen basis.

\section{Preconditioners}\label{sec:solver}
In this section we provide details on our solver and preconditioner choice, as well as the technical tools needed for the numerical analysis of the preconditioned system.

It is known that the convergence of the preconditioned conjugate gradient method for symmetric real operators depends on the condition number of the preconditioned matrix only.
Thus, if we find a preconditioner such that the this condition number is independent of $h$ and of the parameters of the equation,
the number of iterations required for convergence to a certain error is independent of them as well. We will estimate the condition number of the additive Schwarz method by estimating the smallest and largest eigenvalues $c_{\text{ad}}$ and $C_{\text{ad}}$ as
\begin{align*}
c_{\text{ad}} = \inf_{v \ne 0} \frac{\mathcal{A}_h(\mathcal{P}_\text{ad} v,v)}{\|v\|_{\mathcal{A}_h}^2}, && \text{and} && C_{\text{ad}} = \sup_{v \ne 0} \frac{\mathcal{A}_h(\mathcal{P}_\text{ad} v,v)}{\|v\|_{\mathcal{A}_h}^2}.
\end{align*}
For the rest of the preconditioners, we will estimate the norm of the error propagation operator of a preconditioned Richardson iteration.

\subsection{Schwarz preconditioners} \label{SchwarzPreconditioners}
We choose Schwarz preconditioners for which there is a well-known framework and theory for symmetric positive definite problems (see \cite{ToselliWidlund2005,BrennerScott2002,SmithBjorstadGropp1996,FengKarakashian2001}). The following sections provide the definitions needed to prove convergence estimates in an abstract formulation.

Let $V_j$ for $j=0,1,2,\dots,J$ be Hilbert spaces with norms $\|\cdot\|_{V_j}$, where $V_0$ is used to denote the so-called \emph{coarse space} in a domain decomposition context. For $j=0,1,2,\dots,J$, let
\begin{align*}
\mathcal{R}_j^\intercal:V_j \rightarrow V_h
\end{align*}
denote \emph{prolongation operators} for which there holds
\begin{align*}
\mathcal{R}_j^\intercal V_j \subset V &\text{, and} & V = \sum_{j=0}^J \mathcal{R}_j^\intercal V_j, & & \text{ for } j = 0,1,2,\dots,J.
\end{align*}
Here $\mathcal{R}_j^\intercal V_j$ is the range of the linear operator $\mathcal{R}_j^\intercal$.

Associated with each local space $V_j$ for $j=1,2,\ldots,J$, we introduce local discrete bilinear forms $\mathcal{A}_j(\cdot,\cdot)$, defined on $V_j \times V_j$, as the restriction of global discrete bilinear form $\mathcal{A}_h(\cdot,\cdot)$ on $V_j \times V_j$, with $\| v_j \|_{\mathcal{A}_j}^2=\mathcal{A}_j(v_j,v_j)$.

For the coarse space $V_0$ we use the rediscretization of the problem on the coarse mesh, namely a bilinear form \changed{$\mathcal{A}_0(\cdot,\cdot)$} with a penalty parameter inversely proportional to the diameter of the coarse cells $H$, instead of the inherited coarse space obtained by restriction to $V_0 \times V_0$. \changed{For any fixed $v \in V_0$, we define a \emph{projection-like} operator $\widetilde{\mathcal{P}}_0 v \in V_0$ by
\begin{align*}
\mathcal{A}_0(\widetilde{\mathcal{P}}_0 v,w_0) :=& \mathcal{A}_h(v,\mathcal{R}_0^\intercal w_0), && \forall w_0 \in V_0.
\end{align*}
and the composite operator $\mathcal{P}_0 := \mathcal{R}_0^\intercal \widetilde{\mathcal{P}}_0$.}

The convergence analysis of our method follows the standard framework for subspace correction methods, see for instance~\cite{SmithBjorstadGropp1996,ToselliWidlund2005}, which is based on three main assumptions:

\begin{assumption}[Stable decomposition]\label{ass:energystability}
  The spaces $\left\{V_j \right\}$ are said to provide a stable decomposition if there exists a constant $C_V$ such that each $v \in V_h$ admits a decomposition
\begin{align*}
v = \sum_{j=0}^J \mathcal{R}_j^\intercal v_j,
\end{align*}
with $v_j \in V_j$ such that
\begin{align*}
\sum_{j=0}^J \| v_j \|_{\mathcal{A}_j}^2 \le& C_V \| v \|_{\mathcal{A}_h}^2,
\end{align*}
where $\| v \|_{\mathcal{A}_h}^2=\mathcal{A}_h(v,v)$ and $\| v \|_{\mathcal{A}_j}^2$ accordingly.

If $v \in \text{range}\left(\mathcal{I} - \mathcal{P}_0\right)$, $v \in V_h$ admits a stable decomposition without including the coarse space as follows (see \cite[p.49]{ToselliWidlund2005})
\begin{align*}
\sum_{j=1}^J \| v_j \|_{\mathcal{A}_j}^2 \le& C_V \| v \|_{\mathcal{A}_h}^2.
\end{align*}
\end{assumption}

\begin{assumption}[Strengthened Cauchy-Schwarz inequality]\label{ass:strengthenedcauchyschwarz}
There exist constants $\theta_j \in [0,1]$ for $i,j = 0,1,2,\ldots,J$ such that
\begin{align*}
\mathcal{A}_h(\mathcal{R}_i^\intercal v_i, \mathcal{R}_j^\intercal v_j) \le&\theta_{ij} \mathcal{A}_h(\mathcal{R}_i^\intercal v_i,\mathcal{R}_i^\intercal v_i)^\frac{1}{2} \mathcal{A}_h(\mathcal{R}_j^\intercal v_j,\mathcal{R}_j^\intercal v_j)^\frac{1}{2}, \hspace{1cm} \forall v_i \in V_i ,v_j \in V_j.
\end{align*}
We will denote the spectral radius of $\Theta = \{ \theta_{ij} \}$ by $\rho(\Theta)$.
\end{assumption}

\begin{assumption}[Local stability]\label{ass:localcompatibility}
There exists $\omega \in [1,2)$ such that 
\begin{align*}
\mathcal{A}_h(\mathcal{R}_j^\intercal v_j,\mathcal{R}_j^\intercal v_j) \le \omega \mathcal{A}_j(v_j,v_j) & & \forall v_j \in V_j. 
\end{align*}
\end{assumption}

We now introduce a set of \emph{projection-like} operators $\widetilde{\mathcal{P}}_j : V_h \rightarrow V_j$ for $j=0,1,2,\ldots,J$. These projection-like operators will serve as the building blocks for the construction of Schwarz methods. For any fixed $v \in V_h$, define $\widetilde{\mathcal{P}}_j v \in V_j$ by
\begin{align*}
\mathcal{A}_j(\widetilde{\mathcal{P}}_j v,w_j) :=& \mathcal{A}_h(v,\mathcal{R}_j^\intercal w_j), && \forall w_j \in V_j.
\end{align*}

We note that the well posedness of the global problem ensures $\widetilde{\mathcal{P}}_j$ is well defined for $j=0,1,2,\ldots,J$. To map the elements of $V_j$ into the global discrete space $V_h$, we employ the \emph{prolongation operator} $\mathcal{R}_j^\intercal$ and define the composite operator
\begin{align*}
\mathcal{P}_j := \mathcal{R}_j^\intercal \circ \widetilde{\mathcal{P}}_j, && \text{for } j=0,1,2,\ldots,J.
\end{align*}
Trivially, we have $\mathcal{P}_j:V_h \rightarrow V_h$ for $j=0,1,2,\ldots,J$.

After these preparations, we can write the operator $\mathcal{A}_h$ preconditioned with the additive Schwarz method as
\begin{align*}
\mathcal{P}_{\text{ad}} :=& \mathcal{P}_0 + \mathcal{P}_1 + \mathcal{P}_2 + \dots + \mathcal{P}_J.
\end{align*}

To facilitate the comprehension of the method with respect to its implementation, we write the additive operator in a more explicit form. We use the operator notation for the bilinear forms $\mathcal{A}_h$ and $\mathcal{A}_j$ to obtain the following expression for the local projections
\begin{align*}
\mathcal{A}_j \widetilde{\mathcal{P}}_j v :=& \mathcal{R}_j \mathcal{A}_h v, && \forall v \in V_h.
\end{align*}

Thus,
\begin{align*}
\widetilde{\mathcal{P}}_j =& \mathcal{A}_j^{-1} \mathcal{R}_j^{\intercal} \mathcal{A}_h, && \text{and} && \mathcal{P}_j = \mathcal{R}_j^\intercal \mathcal{A}_j^{-1} \mathcal{R}_j \mathcal{A}_h, 
\end{align*}

Finally, our additive Schwarz preconditioned system reads
\begin{align*}
\mathcal{P}_{\text{ad}} &= \mathcal{R}_0^\intercal \mathcal{A}_0^{-1} \mathcal{R}_0 \mathcal{A}_h + \sum_{j=1}^J \mathcal{R}_j^\intercal \mathcal{A}_j^{-1} \mathcal{R}_j \mathcal{A}_h.
\end{align*}

While the additive version applies all subspace corrections at once and adds them in the end, the multiplicative version applies them successively. It can be defined easily by the \emph{error propagation} operator
\begin{align*}
\mathcal{E}_\text{mu} = \left(\mathcal{I} - \mathcal{P}_N\right) \circ \left(\mathcal{I} - \mathcal{P}_{N-1} \right) \circ \dots \circ \left(\mathcal{I} - \mathcal{P}_0\right),
\end{align*}
where $\mathcal{I}$ denotes the identity operator on $V$. Using $\mathcal{E}_\text{mu} $ we define the multiplicative Schwarz preconditioner as
\begin{align*} 
\mathcal{P}_\text{mu} = \mathcal{I} - \mathcal{E}_\text{mu},
\end{align*}
where $\mathcal{I}$ denotes the identity operator on $V_h$.

Finally, we consider the symmetric, hybrid version, which is additive with respect to the subdomain spaces, but applies the coarse grid correction in a multiplicative way:
\begin{gather} 
\mathcal{P}_\text{hy} = \mathcal{I}  - \left(\mathcal{I} - \sum_{i = 1}^N \mathcal{P}_i\right)(\mathcal{I} - \mathcal{P}_0) \left(\mathcal{I} - \sum_{i = 1}^N \mathcal{P}_i\right).
\end{gather}

Following, we will prove convergence estimates for the operators $\mathcal{P}_{\text{ad}}$, $\mathcal{P}_{\text{hy}}$ and $\mathcal{P}_{\text{mu}}$. For $\mathcal{P}_{\text{ad}}$ we estimate the condition number, for $\mathcal{P}_{\text{mu}}$ we bound the error operator of a preconditioned Richardson iteration and for $\mathcal{P}_{\text{hy}}$ we defer the proof to section \ref{sec:multigrid}, where we study multigrid preconditioners, from which $\mathcal{P}_{\text{hy}}$ is a special case.

We use the general abstract convergence theory of Schwarz methods given in \cite[\S2]{ToselliWidlund2005}. We quote the convergence results below.

\begin{thm}\label{thm:additiveestimate}
Let the assumptions \ref{ass:energystability}, \ref{ass:strengthenedcauchyschwarz} and \ref{ass:localcompatibility} hold, then the following bounds hold for the additive Schwarz preconditioned system
\begin{align*}
c_\text{ad} \ge& \frac{1}{C_V}, && C_\text{ad} \le \omega(\rho(\Theta)+ 1)
\end{align*}
Where $c_\text{ad}$ and $C_\text{ad}$ are the smallest and largest eigenvalues of the preconditioned system, respectively.
\end{thm}
\begin{proof}
See \cite[\S2.3]{ToselliWidlund2005}.
\end{proof}

\begin{thm}\label{thm:hybridestimate}
Let the assumptions \ref{ass:energystability}, \ref{ass:strengthenedcauchyschwarz} and \ref{ass:localcompatibility} hold, then the following bounds hold for the hybrid Schwarz preconditioned system
\begin{align*}
\left|\mathcal{A}_h\left(\left[\mathcal{I}-\mathcal{P}_\text{hy} \right]v,v \right)\right| \le c_\text{MG} \mathcal{A}_h(v,v), && \forall v \in V_h,
\end{align*}
where $c_\text{MG}$ is a constant independent of $h$ and $\varepsilon$.
\end{thm}
\begin{proof}
We defer this proof to section \ref{sec:multigrid}, as it is a special case of a multigrid preconditioner and as such its convergence estimate is given in theorem \ref{thm:multigrid}.
\end{proof}

The multiplicative operator is not symmetric and we will consider a simple Richardson iteration applied to the corresponding preconditioned system and provide an upper bound for the norm of the error propagation operator.
\begin{thm}\label{thm:multiplicativeestimate}
Let the assumptions in definitions \ref{ass:energystability}, \ref{ass:strengthenedcauchyschwarz} and \ref{ass:localcompatibility} hold, then the following bounds hold for the multiplicative Schwarz preconditioned system
\begin{align*}
\left\| \mathcal{E}_\text{mu} \right\| \le 1 - \frac{2 - \omega}{\left(2 \max\{1,\omega^2\} \rho^2(\Theta) + 1\right)C_V} \le 1
\end{align*}
\end{thm}
\begin{proof}
See \cite[\S2.3]{ToselliWidlund2005}.
\end{proof}

\subsection{Application to the discrete problem}

In this section we define the Schwarz method for  the discrete problem in equation~\eqref{eqn:IPDiffusion} and verify that assumptions \ref{ass:energystability}, \ref{ass:strengthenedcauchyschwarz} and \ref{ass:localcompatibility} apply.

After enumerating the cells $\cell_j\in\mesh_h$ for $j=1,\ldots,J$, we choose the local spaces $V_j = V(\cell_j) = \mathbb{Q}_p^G(\kappa_j)$, together with the coarse space $V_0$, defined on $\mesh_H$.
We remark that we are using a nonoverlapping subdivision order to define the direct decomposition $V_h = \oplus_{j=1}^J \mathcal{R}_j^\intercal V_j$, where $\mathcal{R}_j^\intercal\colon V_j \to V_h$ is the simple injection. Similarly, for $v\in V_h$, $\mathcal{R}_jv(x) = v(x)$ if $x\in \cell_j$ and zero otherwise.

Following, we list three standard results from \cite{FengKarakashian2001} that we need for our proof.

For any $v \in V_D = \prod_{K \in \mesh_H} \mathcal{V}(K)$, there holds the trace inequality (see \cite[Lemma 3.1]{FengKarakashian2001})
\begin{align}\label{eqn:TraceInequality}
\|v\|_{\mathcal{H}(\partial D)}^2 \le c \left[\frac{1}{H} \| v \|_{\mathcal{H}(D)}^2 + H \| v \|_{\mathcal{V}(D)}^2 \right].
\end{align}

Suppose $D$ is a convex domain. For any $v \in V_D$, let $\overline{u} = \frac{1}{\text{meas}(D)} \int_D v dx$ be the average value of $v$ over $D$. Then we can write a Poincar\'e inequality as follows (see \cite[Lemma 3.2]{FengKarakashian2001})
\begin{align*}
\| v - \overline{v} \|_{\mathcal{H}(D)} \le c \text{ diam}(D) \|u\|_{\mathcal{V}(D)} \text{ on } D.
\end{align*}
In particular, if $D \in \mesh_H$
\begin{align}\label{eqn:InterpolationEstimate}
\| v - \overline{v} \|_{\mathcal{H}(D)} \le c H \|u\|_{\mathcal{V}(D)} \text{ on } D.
\end{align}

Let $v,w \in V_h$, let $v_j,w_j \in V_j$, $j=1,\dots,J$, be given (uniquely) by $v=\sum_{j=1}^J v_j$, $w = \sum_{j=1}^J w_j$. Then the following identity holds (see \cite[Lemma 3.3]{FengKarakashian2001})
\begin{align}\label{eqn:BlockDiagInt}
a_h(v,w) = \sum_{j=1}^J a_j(v_j,w_j) + I(v,w),
\end{align}
where $I(\cdot,\cdot):V_h \times V_h \rightarrow \mathbb{R}$ comprises all terms located outside the block diagonal of the bilinear from $a_h(v,w)$, connecting different subdomains.

We then obtain the following interface estimate for cell-wise subdomains
\begin{lem} \label{lem:InterfaceEstimate}
There exists a constant $c$ such that
\begin{align*}
|I(v,v)| \le c \left[\frac{1}{h^2} \sum_{K \in \mesh_h} \|v\|_{\mathcal{H}(K)}^2 + a_h(v,v) \right].
\end{align*}
\end{lem}
\begin{proof}
We extend the result in \cite[Lemma 4.3]{FengKarakashian2001}. The following estimate, from \cite[Eq. (4.20)]{FengKarakashian2001}, holds when using cell wise subdomains
\begin{align*}
|I(v,v)| \le c \left(a_h(v,v) + \frac{1}{h} \sum_{F \in \left(\mathbb{F}_h^I \cup \mathbb{F}_h^B\right)} \|v\|_{\mathcal{H}(F)}^2 \right),
\end{align*}
where $\| \cdot \|_{F}$ is the $L^2$-inner product on the faces of cell $K$ of the fine mesh.

Using the trace inequality $\|v\|_{\mathcal{H}(F)}^2 = c \left[\frac{1}{h} \|v\|_{\mathcal{H}(K)}^2 + h \|\nabla v\|_{\mathcal{H}(K)}^2 \right]$ from \cite[Eq. (3.9)]{FengKarakashian2001}, we obtain
\begin{align*}
|I(v,v)| \le c \left(a_h(v,v) + \frac{1}{h} \sum_{K \in \mesh_h} \left[ \frac{1}{h} \|v\|_{\mathcal{H}(K)}^2 + h \|\nabla v\|_{\mathcal{H}(K)}^2 \right] \right).
\end{align*}
The result follows from observing that $\sum_{K \in \mesh_h} \|\nabla v\|_{\mathcal{H}(K)}^2 \le c~a_h(v,v)$.
\end{proof}

Finally, we concentrate on a stable decomposition. The convergence theory from \cite{FengKarakashian2001} requires that the subdomains used for the Schwarz method are at least the same size as the cells in the coarse mesh. Recently, an extension has been published in \cite{Dryja2016} to include the case of cell-wise subdomains, however, the proof uses $P_1$ nonconforming interpolant and enriching operators for simplices \cite{Brenner2003}.

We achieve a stable decomposition, by a close examination of the proof in \cite{FengKarakashian2001}, which holds for simplices, quadrilaterals, and hexahedra. In particular, it does not require auxiliary spaces with continuity assumptions like Crouzeix-Raviart.

\begin{lem}\label{lem:laplacianstabledecomposition}
Every $v \in V_h$ admits a decomposition of the form $v = \sum_{j=0}^J \mathcal{R}_j^\intercal V_j$, $v_j \in V_j$, $j = 0,\dots,J$ which satisfies the bound
\begin{gather}
\label{eq:stable1}
\sum_{i=0}^J a_j\left(v_j,v_j\right) \le C_{V,\Delta} a\left(v,v\right),
\end{gather}
with $C_{V,\Delta} = \mathcal{O}\left(\frac{H^2}{h^2}\right)$, where $h$ and $H$ denote the cell diameters used in the fine and coarse meshes respectively.
\end{lem}
\begin{proof}
Let $v_0 \in V_0$ be the piecewise constant average of $v$ on the coarse mesh $\mesh_H$, let $w = v - v_0$. We decompose $w$ in nonoverlapping cell-wise subdomains as follows
\begin{align*}
w = \sum_{j=1}^J v_j,
\end{align*}
where $v_1,\dots,v_J$ are uniquely determined. From equation \eqref{eqn:BlockDiagInt} we have
\begin{align*}
a_h(w,w) =& \sum_{j=1}^J a_j(v_j,v_j) + I(w,w), \\
\intertext{or equivalently,}
a_h(v-\mathcal{R}_0^\intercal v_0,v-\mathcal{R}_0^\intercal v_0) =& \sum_{j=1}^J a_j(v_j,v_j) + I(v-\mathcal{R}_0^\intercal v_0,v-\mathcal{R}_0^\intercal v_0),
\end{align*}
Reordering and estimating the interface term by its absolute value we obtain
\begin{gather}
\label{eq:stable2}
\sum_{j=1}^J a_j(v_j,v_j) \le a_h(v-\mathcal{R}_0^\intercal v_0,v-\mathcal{R}_0^\intercal v_0) + \left|I(v-\mathcal{R}_0^\intercal v_0,v-\mathcal{R}_0^\intercal v_0)\right|,
\end{gather}
using lemma \ref{lem:InterfaceEstimate} we have
\begin{align}
\begin{aligned}\label{eq:stable3}
\sum_{j=1}^J a_j(v_j,v_j) \le& c \left( a\left(v-\mathcal{R}_0^\intercal v_0,v-\mathcal{R}_0^\intercal v_0\right) + \frac{1}{h^2} \sum_{K \in \mesh_h} \|v - \mathcal{R}_0^\intercal v_0\|_{\mathcal{H}(K)}^2  \right)\\
\le& c \left( \left(a_h(v,v)^{1/2} + a\left(\mathcal{R}_0^\intercal v_0,\mathcal{R}_0^\intercal v_0\right)^{1/2} \right)^2 \right.\\
&\left.+ \frac{1}{h^2} \sum_{D \in \mesh_H} \|v - \mathcal{R}_0^\intercal v_0\|_{\mathcal{H}(D)}^2  \right),
\end{aligned}
\end{align}
where we used Minkowsky's inequality and we regrouped the $L^2$ inner products.

We expand the first term and use equation \eqref{eqn:InterpolationEstimate} to obtain
\begin{align*}
\sum_{j=1}^J a_j(v_j,v_j) \le &c \left( a_h(v,v) + 2 a_h(v,v)^{1/2} a_h(\mathcal{R}_0^\intercal v_0,\mathcal{R}_0^\intercal v_0)^{1/2} \right.\\
&\left. + a_h(\mathcal{R}_0^\intercal v_0,\mathcal{R}_0^\intercal v_0) + \frac{H^2}{h^2} \|v\|_\mathcal{V}^2  \right) \\
\le& c \left(2 a_h(v,v) + 2 a_h(\mathcal{R}_0^\intercal v_0,\mathcal{R}_0^\intercal v_0) + \frac{H^2}{h^2} a_h(v,v) \right),
\end{align*}
where we used Young's inequality and coercivity of $a_h(\cdot,\cdot)$.

Finally, including the coarse space we achieve
\begin{align*}
\sum_{j=0}^J a_j(v_j,v_j) \le& c \left(a_0(v_0,v_0) + a_h(\mathcal{R}_0^\intercal v_0,\mathcal{R}_0^\intercal v_0) + \frac{H^2}{h^2} a_h(v,v)\right).
\end{align*}

It remains to bound $a_h(\mathcal{R}_0^\intercal v_0,\mathcal{R}_0^\intercal v_0)$ in a way such that the estimate is independent of the usage of cell-wise or larger subdomains and a constant $\mathcal{O}(\frac{H}{h})$ is achieved as we show below. Since $v_0$ is piecewise constant on $\mesh_H$, and hence also on $\mesh_h$,
\begin{align} \label{eqn:globalcoarse}
\begin{aligned}
a_h(\mathcal{R}_0^\intercal v_0,\mathcal{R}_0^\intercal v_0) =& \ip \sum_{F \in \mathbb{F}_h^I} \frac{1}{h} \|\mathcal{R}_0^\intercal v_0^+ - \mathcal{R}_0^\intercal v_0^-\|_{\mathcal{H}(F)}^2 \\
&+ \ip \sum_{F \in \mathbb{F}_h^B} \frac{1}{h} \|\mathcal{R}_0^\intercal v_0^+\|_{\mathcal{H}(F)}^2,
\end{aligned}
\end{align}
where we observe
\begin{align}\label{eqn:localcoarse}
a_0(v_0,v_0) = \frac{h}{H} a_h(\mathcal{R}_0^\intercal v_0,\mathcal{R}_0^\intercal v_0).
\end{align}

Adding and subtracting $v$ in equation \eqref{eqn:globalcoarse} gives
\begin{multline*}
a_h(\mathcal{R}_0^\intercal v_0,\mathcal{R}_0^\intercal v_0) \le c \ip \left( \sum_{F \in \mathbb{F}_h^I} \frac{1}{h} \|(v - \mathcal{R}_0^\intercal v_0)^+ - (v - \mathcal{R}_0^\intercal v_0)^-\|_{\mathcal{H}(F)}^2 \right. \\
+ \sum_{F \in \mathbb{F}_h^B} \frac{1}{h} \|(v - \mathcal{R}_0^\intercal v_0)^+\|_{\mathcal{H}(F)}^2 \\
+ \left. \sum_{F \in \mathbb{F}_h^I} \frac{1}{h} \|v^+ - v^-\|_{\mathcal{H}(F)}^2 + \sum_{F \in \mathbb{F}_h^B} \frac{1}{h} \|v^+\|_{\mathcal{H}(F)}^2 \right).
\end{multline*}
The last two terms are obviously bounded by $a_h(v,v)$. Also, since $u_0$ is piecewise constant on each $D \in \mesh_H$, $\|(v - \mathcal{R}_0^\intercal v_0)^+ - (v - \mathcal{R}_0^\intercal v_0)^-\|_{\mathcal{H}(F)} = \|v^+ - v^-\|_{\mathcal{H}(F)}$ whenever $F$ is in the interior of some $D \in \mesh_H$. Thus,
\begin{multline*}
\sum_{F \in \mathbb{F}_h^I} \frac{1}{h} \|(v - \mathcal{R}_0^\intercal v_0)^+ - (v - \mathcal{R}_0^\intercal v_0)^-\|_{\mathcal{H}(F)}^2 + \sum_{F \in \mathbb{F}_h^B} \frac{1}{h} \|(v - \mathcal{R}_0^\intercal v_0)^+\|_{\mathcal{H}(F)}^2 \\
= \sum_{D \in \mesh_H} \left( \sum_{F \subset D} \|v^+ - v^-\|_{\mathcal{H}(F)} \right. \\
\left. + \sum_{F \in \partial D} \frac{1}{h} \|(v - \mathcal{R}_0^\intercal v_0)^+ - (v - \mathcal{R}_0^\intercal v_0)^-\|_{\mathcal{H}(F)}^2 \right. \\
\left.+ \sum_{F \subset \partial D \in \mathbb{F}_h^B} \frac{1}{h} \|(v - \mathcal{R}_0^\intercal v_0)^+\|_{\mathcal{H}(F)}^2 \right)\\
\le c a_h(v,v) + c \sum_{D \in \mesh_H} \frac{1}{h} \| v - \mathcal{R}_0^\intercal v_0 \|_{\mathcal{H}(\partial D)}^2.
\end{multline*}
Now using the trace inequality in equation \eqref{eqn:TraceInequality}, we obtain
\begin{align*}
\sum_{D \in \mesh_H} \frac{1}{h} \| v - \mathcal{R}_0^\intercal v_0 \|_{\mathcal{H}(\partial D)}^2 \le c \sum_{D \in \mesh_H} \frac{1}{h} \left[\frac{1}{H} \| v - \mathcal{R}_0^\intercal v_0 \|_{\mathcal{H}(D)}^2 + H \| v - \mathcal{R}_0^\intercal v_0 \|_{\mathcal{V}(D)}^2 \right].
\end{align*}
Also note that $\| v - \mathcal{R}_0^\intercal v_0 \|_{\mathcal{V}(D)}^2 = \| v \|_{\mathcal{V}(D)}^2$. Hence, applying the approximation result from equation \eqref{eqn:InterpolationEstimate} to $\| v - v_0 \|_{\mathcal{H}(D)}$ we obtain
\begin{align} \label{eqn:coarseestimate}
a_h(\mathcal{R}_0^\intercal v_0,\mathcal{R}_0^\intercal v_0) \le c \frac{H}{h} a_h(v,v),
\end{align}
therefore, using this result on equation \eqref{eqn:localcoarse} we see that $a_0(v_0,v_0) \le c a_h(v,v)$ and the result is achieved.
\end{proof}

\begin{lem}\label{lem:energystability}
\emph{Stable decomposition}. The spaces $V_j$ provide a stable decomposition of $V$, with respect to the bilinear form $\mathcal{A}_h(\cdot,\cdot)$, in the sense of assumption~\ref{ass:energystability}.
\end{lem}
\begin{proof}
Let $C_{V,\Delta}$ be the stable decomposition constant for the Laplacian, as deduced in lemma \ref{lem:laplacianstabledecomposition}, we then have
\begin{align*}
	\sum_{i=0}^J \mathcal{A}_j\left(v_j,v_j\right)
    &= \sum_{i=0}^J \left\{ a_j\left(v_j,v_j\right) + \frac{1}{\varepsilon} \left(\boldsymbol{\Sigma} v_j , v_i\right)_{\mathcal{H}} \right\}
    \\&= \sum_{i=0}^J a_j\left(v_j,v_j\right) + \frac{1}{\varepsilon} \left(\boldsymbol{\Sigma} v, v\right)_{\mathcal{H}}
    \\&\le C_{V,\Delta} a\left(v,v\right) + \frac{1}{\varepsilon} \left(\boldsymbol{\Sigma} v, v\right)_{\mathcal{H}}
    \\& \le \max\left\{ C_{V,\Delta},1 \right\} \mathcal{A}_h\left(v,v\right).
\end{align*}
It follows that the $V_j$ decomposition for our reaction-diffusion problem is \emph{energy stable} with $C_V = C_{V,\Delta} = \mathcal{O} \left(\frac{H^2}{h^2} \right)$ where $H$ and $h$ are the largest and smallest cell diameters respectively.
\end{proof}

\begin{lem}\label{prop:strengthenedcauchyschwarz}
There exists a strengthened Cauchy-Schwarz inequality in the sense of definition \ref{ass:strengthenedcauchyschwarz}.
\end{lem}
\begin{proof}
(See \cite[\S4.2]{FengKarakashian2001}). Verifying this inequality consists of obtaining a bound for the spectral radius $\rho(\Theta)$ of the $J \times J$ matrix $\Theta = \left[\theta_{ij}\right]_{j=0}^J$.

That such values exist is a consequence of the Cauchy-Schwarz inequality. The important thing, however, is to obtain a small bound on $\rho$. To do so, we observe that $a_h(\mathcal{R}_i^\intercal v_i,\mathcal{R}_j^\intercal v_j) = 0$ if the supports of $v_i$ and $v_j$ do not share a face $f_{ij}$. For the remaining cases, we take $\theta_{ij}=1$. It follows at once from Gershgorin's circle theorem that
\begin{equation*}
\rho(\Theta) \le \max_m \text{card}\left\{k | f_{mk} \neq 0 \text{ almost everywhere} \right\} + 1 \hspace{1cm} f_{mk} \in {\mathbb{F}_h^I \cup \mathbb{F}_h^B}
\end{equation*}
i.e., $\rho(\Theta)$ is bounded by $1$ plus the maximum number of adjacent subdomains a given subdomain can have. In practice this number $4$ in 2D and $6$ in 3D. Even for ``unusual" subdomain partitions, this number is not expected to be large.
\end{proof}

\begin{lem}[Local Stability]\label{lem:localstability}
There holds
\begin{align*}
\mathcal{A}_h(\mathcal{R}_j^\intercal v_j,\mathcal{R}_j^\intercal v_j) \le \omega \mathcal{A}_j(v_j,v_j) & & \forall v_j \in V_j,
\end{align*}
where $\omega=\alpha \frac{H}{h}$ for $\alpha<1$.
\end{lem}
\begin{proof}
In the case of exact local solvers $\omega=1$, in our case the coarse bilinear form uses a penalty parameter depending on the cell diameter of the coarse mesh.
Observing the bilinear form \eqref{eqn:DGIPoriginal}, we see that for our coarse space bilinear form it holds
\begin{align*}
\mathcal{A}_h(\mathcal{R}_0^\intercal v_0,\mathcal{R}_0^\intercal v_0) \le \frac{H}{h} \mathcal{A}_0(v_0,v_0)
\end{align*}
and hence our local stability constant would be $\omega=\frac{H}{h}$, however, this would violate assumption \ref{ass:localcompatibility}. To remediate this, we scale the bilinear forms with a relaxation parameter $\alpha$ in order to accomplish the upper bound requested.

We can always introduce such a relaxation parameter but we are not free to scale the local bilinear forms arbitrarily in order to decrease $C_V$ from lemma \ref{lem:energystability}; a small value of $\omega$ means that corrections of the error are small. In such a case $C_V$ will necessarily be large (see \cite[p. 155]{SmithBjorstadGropp1996} and \cite[p. 41]{ToselliWidlund2005}).

Finally, we remark that this is only needed for our proofs, since in practice such relaxation parameter is not needed.
\end{proof}

\subsection{Multigrid V-cycle preconditioner}\label{sec:multigrid}
The preconditioners developed in the previous section are easily implemented as smoothers for multigrid preconditioners. In this section we provide convergence estimates for the multigrid $V$-cycle.

Let $\{\mesh\}_{\ell=0,\dots,L}$ be a hierarchy of meshes of quadrilateral and hexahedral cells in two and three dimensions, respectively. In view of multilevel methods, the index $\ell$ refers to the mesh level defined as follows: let a coarse mesh $\mesh_0$ be given. The mesh hierarchy is defined recursively, such that the cells of $\mesh_{\ell+1}$ are obtained by splitting each cell of $\mesh_\ell$ into $2^d$ children by connecting edge and face midpoints (refinement). These meshes are nested in the sense that every cell of $\mesh_\ell$ is equal to the union of its four (respectively eight) children. We define the mesh size $h_\ell$ as the maximum of the diameters of the cells of $\mesh_\ell$. Due to the refinement process, we have $h_\ell \approx 2^{-1} h_{\ell-1}$.

Due to the nestedness of mesh cells, the finite element spaces associated with these meshes are nested as well:
\begin{equation*}
V_0 \subset V_1 \subset \dots \subset V_L.
\end{equation*}
We introduce the $L^2$-projections $\mathcal{Q}_{\ell-1}$ and embedding operators $\mathcal{Q}_{\ell-1}^\intercal$
\begin{align*}
\mathcal{Q}_{\ell-1}&:V_{\ell} \rightarrow V_{\ell-1}, \\
\mathcal{Q}_{\ell-1}^\intercal&:V_{\ell-1} \rightarrow V_{\ell},
\end{align*}
such that
\begin{gather}
  \left(\mathcal{Q}_{\ell-1} v_{\ell},w_{\ell-1}\right)_{\mathcal{H}}
  = \left(v_{\ell},\mathcal{Q}_{\ell-1}^\intercal w_{\ell-1}\right)_{\mathcal{H}}
  \qquad \forall v_{\ell-1} \in V_{\ell-1}, w_{\ell-1} \in V_{\ell-1}
\end{gather}

Let $\mathcal{A}_{\ell}(\cdot,\cdot)$ be the bilinear form defined in equation~\eqref{eqn:IPDiffusion} on the mesh $\mesh_{\ell}$.
We define the operator $\mathcal{A}_{\ell}:V_{\ell} \longrightarrow V_{\ell}$ such that $\mathcal{A}_{\ell}(u_{\ell},v_{\ell}) = (\mathcal{A}_{\ell} u_{\ell}, v_{\ell})_{\mathcal{H}}$.

For the rest of the paper, we will redefine the operators $\mathcal{P}$ used in the 2-level analysis as follows: $\mathcal{P}_{\ell-1}$ is what used to be the coarse grid solver $\mathcal{P}_{0}$, while $\mathcal{P}_{\ell,j}$ represent the projections onto the subdomain spaces $V_j = V_{\ell,j}$ on mesh level $\ell$. There holds $\mathcal{A}_{\ell-1} \mathcal{P}_{\ell-1} = \mathcal{Q}_{\ell-1} \mathcal{A}_{\ell}$.

Let $\mathcal{B}_\ell$ be a smoother defined as the preconditioning operator on the preconditioned systems presented in \S\ref{SchwarzPreconditioners} without including the coarse space, i.e.,
\begin{align*}
\mathcal{B}_{\ell,\text{ad}} =& \sum_{i=1}^{N_\ell} \mathcal{P}_{\ell,i} \mathcal{A}_\ell^{-1} = \sum_{i=1}^{N_\ell} \mathcal{R}_{\ell,i}^\intercal \mathcal{A}_{\ell,i}^{-1} \mathcal{R}_{\ell,i}& \text{and} &&\mathcal{B}_{\ell,\text{mu}} =& \left(\mathcal{I} - \prod_{i=N_\ell}^{1} \mathcal{P}_{\ell,i} \right) \mathcal{A}_\ell^{-1}.
\end{align*}

We define the multigrid preconditioner $\mathcal{M}_L$ by induction. Let $\mathcal{M}_0 = \mathcal{A}_0^{-1}$. For $1 \le \ell \le L$ we define the action $\mathcal{M}_\ell g$ of $\mathcal{M}_\ell$ on a vector $g\in V_\ell$ in terms of $\mathcal{M}_{\ell-1}$:
\begin{enumerate}
\item Let $x_0 = 0$.
\item Define $x_i$ for $i=1,\dots,m$ by $m$ pre-smoothing steps
\begin{align*}
x_i = x_{i-1} + \mathcal{B}_\ell \left(g - \mathcal{A}_\ell x_{i-1} \right).
\end{align*}

\item Define $y_0$ by coarse grid correction
\begin{align*}
y_0 = x_m + \mathcal{Q}_{\ell-1}^\intercal \mathcal{M}_{\ell-1} \mathcal{Q}_{\ell-1} \left(g - \mathcal{A}_\ell x_m \right).
\end{align*}
\item define $y_i$ for $i=1,\dots,m$ by $m$ post-smoothing steps
\begin{align*}
y_i = y_{i-1} + \mathcal{B}_\ell \left(g - \mathcal{A}_\ell x_{i-1} \right).
\end{align*}
\item Let $\mathcal{M}_\ell g = y_{m}$.
\end{enumerate}
Our analysis of the multigrid algorithm follows~\cite{Duan2007}, since we have noninherited forms. There, convergence is proven in an abstract framework under the following three assumptions:
\begin{assumption}[Stability]\label{ass:mgstability}
There is a constant $C_Q > 0$ such that for all levels $\ell=2,\dots,L$ and all $v_\ell \in V_\ell$
\begin{gather}
\mathcal{A}_\ell\left(\left[\mathcal{I}_\ell - \mathcal{Q}_{\ell-1}^\intercal \mathcal{P}_{\ell-1}\right] v_\ell,\left[\mathcal{I}_\ell - \mathcal{Q}_{\ell-1}^\intercal \mathcal{P}_{\ell-1}\right] v_\ell\right) \le C_Q \mathcal{A}_\ell(v_\ell,v_\ell).
\end{gather}
\end{assumption}
\begin{assumption}[Regularity-approximation property]\label{ass:mgregularity}
There is a constant $C_1 > 0$, such that for all levels $\ell=2,\dots,L$ and all $v_\ell \in V_\ell$
\begin{gather}
\mathcal{A}_\ell\left(\left[\mathcal{I}_\ell - \mathcal{Q}_{\ell-1}^\intercal \mathcal{P}_{\ell-1}\right] v_\ell,v_\ell\right) \le C_1 \frac{\left\| \mathcal{A}_\ell v_\ell \right\|_{L^2}^2}{\Lambda_\ell}.
\end{gather}
where $\Lambda_\ell$ is the maximum eigenvalue of $\mathcal{A}_\ell$.
\end{assumption}
\begin{assumption}[Smoothing property]\label{ass:mgsmoothing}
There is a constant $C_R > 0$ such that for all levels $\ell=2,\dots,L$ and all $v_\ell \in V_\ell$
\begin{align*}
\frac{\left\|v_\ell \right\|_{L^2}^2}{\Lambda_\ell} \le C_R (\overline{R} v_\ell,v_\ell),
\end{align*}
where $\overline{R}=(\mathcal{I}-\mathcal{K}_\ell^2)\mathcal{A}_\ell^{-1}$ and $\mathcal{K}_\ell = \mathcal{I} - \mathcal{B}_\ell \mathcal{A}_\ell$.
\end{assumption}

From~\cite{Duan2007} we qoute the estimate for the error propagation operator defined as $\mathcal{I}-\mathcal{M}_\ell\mathcal{A}_\ell$.
\begin{thm}\label{thm:multigrid}
Let assumptions \ref{ass:mgstability}, \ref{ass:mgregularity} and \ref{ass:mgsmoothing} hold. Furthermore, assume $m>2 C_1 C_R$. Then, for all $\ell \ge 0$, there holds
\begin{align*}
\left|\mathcal{A}_\ell\left(\left[\mathcal{I}-\mathcal{M}_\ell\mathcal{A}_\ell\right]v_\ell,v_\ell \right)\right| \le c_\text{MG} \mathcal{A}_\ell(v_\ell,v_\ell), && \forall v_\ell \in V_\ell,
\end{align*}
with
\begin{align*}
c_\text{MG}=\frac{C_1 C_R}{m+C_1 C_R}
\end{align*}
for the two-level method, i.e., $\mathcal{P}_\text{hy}$, and 
\begin{align*}
c_\text{MG}=\frac{C_1 C_R}{m-C_1 C_R}
\end{align*}
for $L > 2$.

\end{thm}
We refer to \cite{Duan2007} for the proof in an abstract setting, we show below that the assumptions apply to our method.

Assumption \ref{ass:mgregularity} is proven in \cite[Th. 9]{AntoniettiPennesi2018} and assumption \ref{ass:mgsmoothing} in \cite[Th. 5.1]{Bramble1993}.

To prove assumption \ref{ass:mgstability}, we use lemma \ref{lem:localstability} as follows
\begin{gather*}
\mathcal{A}_\ell\left(\mathcal{Q}_\ell^\intercal \mathcal{P}_{\ell-1} v_\ell,\mathcal{Q}_\ell^\intercal \mathcal{P}_{\ell-1} v_\ell\right) \le 2 \mathcal{A}_{\ell-1}\left(\mathcal{P}_{\ell-1} v_\ell,\mathcal{P}_{\ell-1} v_\ell\right) \\
\mathcal{A}_\ell\left(\mathcal{Q}_\ell^\intercal \mathcal{P}_{\ell-1} v_\ell,\mathcal{Q}_\ell^\intercal \mathcal{P}_{\ell-1} v_\ell\right) \le 2 \mathcal{A}_\ell\left(v_\ell,\mathcal{Q}_\ell^\intercal \mathcal{P}_{\ell-1} v_\ell\right) \\
\mathcal{A}_\ell\left(\mathcal{Q}_\ell^\intercal \mathcal{P}_{\ell-1} v_\ell,\mathcal{Q}_\ell^\intercal \mathcal{P}_{\ell-1} v_\ell\right) - 2 \mathcal{A}_\ell\left(v_\ell,\mathcal{Q}_\ell^\intercal \mathcal{P}_{\ell-1} v_\ell\right) + \mathcal{A}_\ell\left(v_\ell, v_\ell\right) \le \mathcal{A}_\ell\left(v_\ell, v_\ell \right), 
\end{gather*}
and we deduce
\begin{align}\label{eqn:CQ}
\mathcal{A}_\ell\left(\left[\mathcal{I} - \mathcal{Q}_{\ell-1}^\intercal \mathcal{P}_{\ell-1}\right] v_\ell,\left[\mathcal{I} - \mathcal{Q}_{\ell-1}^\intercal \mathcal{P}_{\ell-1}\right] v_\ell\right) \le \mathcal{A}_\ell\left(v_\ell, v_\ell \right),
\end{align}
hence assumption \ref{ass:mgstability} holds with $C_Q=1$.

We note that the preceding theorem requires $m>1$ for $L > 2$ but as we will see in the next section, $m=1$ suffices for our setting. For completeness, we provide results for $m>1$ as well.

\section{Numerical experiments}\label{sec:numericalexperiments}
Given that some of the smoothers and preconditioners we use are not symmetric, we use a GMRES solver for all our calculations.

\changed{All our experiments are performed on a unit square $\Omega = (0,1) \times (0,1)$, with Dirichlet boundary conditions. As our work is centered on the behavior of the system with respect to the reaction-term, we set all diffusion coefficients to one. Each mesh $\mesh_\ell$ consists of $2^\ell\times 2^\ell$ cells, such that each cell is a square with side $h_\ell=2^{-\ell}$. We use bilinear elements and $\delta_0=2$ and $h_\ell/h_{\ell-1}=1/2$ for all our experiments. When using multigrid V-cycles, the coarsest mesh consists of a single cell.}

\subsection{Poisson's equation}
As a baseline for further experiments we show the results for Poisson's equation using different preconditioners.
\begin{table}[ht!]
\centering
\caption{GMRES iterations for a DG discretization of Poisson's equation using tensor product polynomials of degree 1 and a unit source to reduce the residual by $10^{-8}$ for $\Sigma = 0$. Where U is unpreconditioned; 2AS, 2HS, 2MS are two-level additive, hybrid and multiplicative Schwarz respectively; MGAS, MGMS are multigrid with additive and multiplicative Schwarz smoothers respectively.}
\begin{tabular}{c|cccccc}
levels & U & 2AS & 2HS & 2MS & MGAS & MGMS \\
\hline
2 &  3      &  3 &  3 &  4 &  3 &  4 \\ 
3 &  10     & 10 &  6 &  6 &  6 &  6 \\ 
4 &  22     & 18 &  9 &  7 & 10 &  7 \\ 
5 &  43     & 24 & 11 &  7 & 12 &  8 \\ 
6 &  85     & 26 & 11 &  7 & 13 &  8 \\ 
7 &  $>100$ & 25 & 11 &  7 & 14 &  8 \\
8 &  $>100$ & 25 & 11 &  7 & 14 &  8 
\end{tabular}
\label{tab:unpreconditioned}
\end{table}

We observe that all preconditioners achieve a flat iteration count, albeit with different amount of iterations at very fine levels. Two-level additive Schwarz, for instance, requiring almost double the amount of iterations than multigrid with additive Schwarz preconditioners.

\subsection{2 groups}
In the case of a two group problem, because of the conservation condition of zero column sum and symmetry, all reaction matrices are multiples of
\begin{align*}
\boldsymbol{\Sigma} = \frac1{\varepsilon}\begin{pmatrix} 
1 & -1 \\ 
-1 & 1
\end{pmatrix}.
\end{align*}

We show iteration results in Table \ref{tab:G2S11a}.
\begin{table}
\centering
\caption[]{GMRES Iterations using a source $(1,0) \text{ or } (0,1)$  to reduce the residual by $10^{-8}$, where "max" is the maximum amount of iterations for different $\varepsilon$.}
\begin{tabular}{c|ccccc|c|c|c|c}
& \multicolumn{5}{c|}{MGAS} & MGMS & 2AS & 2HS & 2MS \\
\diagbox{levels}{$\varepsilon$} & 1.0 &  $10^{-1}$ & $10^{-2}$ & $10^{-3}$ & $10^{-4}$ & max & max & max & max \\
\hline
2 &  5 &  5 &  4  &  4  &  4  &  4 & 6  & 5 & 4\\
3 &  8 &  8 &  6  &  6  &  6  &  6 & 14  & 8 & 6 \\
4 &  10 &  10 &  10  & 10  & 10  &  7 & 22 & 10 & 7 \\
5 &  12 &  12 &  12  & 12  & 12  &  8 & 25 & 11 & 7 \\
6 &  13 &  13 &  13  & 13  & 13  &  8 & 25 & 11 & 7 \\
7 &  14 &  14 &  14  & 14  & 14  &  8 & 25 & 11 & 7 \\
8 &  14 &  14 &  14  & 14  & 14  &  8 & 25 & 11 & 7 \\
9 &  14 &  14 &  14  & 14  & 14  &  8 & 25 & 11 & 7
\end{tabular}
\label{tab:G2S11a}
\end{table}

We observe that the iteration count flattens for all methods considered, with very similar numbers to the pure Laplacian problem, indicating that the reaction operator does not affect the results shown in the previous section.
The fact that the results do not improve is explained by the reaction operator having a non-trivial kernel, where we effectively solve for the Laplacian.

\subsection{Multigroup}
We devise a reaction matrix with a \emph{contrast} between coefficients in different groups that is inversely proportional to different powers of $\varepsilon$ as follows:
\begin{align*}
\boldsymbol{\Sigma} = \begin{pmatrix} 
\changed{\Sigma_{1,1}} & -1 & -\varepsilon^{-1} & -\varepsilon^{-2} & -\varepsilon^{-3} & \dots\\ 
-1 & \changed{\Sigma_{2,2}} & -1 & -1 & -1 & \dots\\
-\varepsilon^{-1} & -1 & \changed{\Sigma_{3,3}} & -\varepsilon^{-1} & -\varepsilon^{-2} & \dots\\
-\varepsilon^{-2} & -1 & -\varepsilon^{-1} & \changed{\Sigma_{4,4}} & -\varepsilon^{-1} & \dots\\
-\varepsilon^{-3} & -1 & -\varepsilon^{-2} & -\varepsilon^{-1} & \changed{\Sigma_{5,5}} & \dots\\
\vdots & \vdots & \vdots & \vdots & \vdots & \ddots
\end{pmatrix}
\end{align*}
where \[\Sigma_{\gi,\gi} = -\displaystyle\sum_{\gj \neq \gi} \Sigma_{\gi,\gj} = 1 + \varepsilon^{-1} + \varepsilon^{-2} + \varepsilon^{-3} + \dots\]. We remark that the elements in the diagonal are such that the matrix has zero column sum. We use the top left $5\times5$-block of this matrix as the reaction matrix in the following tests.

Results are shown in Table \ref{tab:G5a}, tests were performed for the sources $(1,0,1,0,1)$, $(0,1,0,1,0)$, $(0,1,1,1,0)$ and $(1,0,0,0,1)$ and we report the maximum iteration count encountered. In this case the columns are shown only up to $\varepsilon=0.01$ to avoid floating point underflow problems. Note, that this involves values of $\varepsilon^{-3}=10^{-6}$.
\begin{table}[ht!]
\centering
\caption[]{GMRES Iterations to reduce the residual by $10^{-8}$ for a 5 groups calculation, where "max" is the maximum amount of iterations over the values of $\varepsilon$ in the left columns.}
\begin{tabular}{c|ccc|c|c|c|c}
& \multicolumn{3}{c|}{MGAS} & MGMS & 2AS & 2HS & 2MS\\
\diagbox{levels}{$\varepsilon$} & 1.0 &  0.1 &  0.01 & max & max & max & max \\
\hline
2 &  5 &  5 &  4  & 4 & 9  & 5 & 4   \\
3 &  8 &  7 &  6  & 6 & 15 & 8 & 6  \\
4 &  10 &  10 &  10 & 7 & 22  & 10 & 7   \\
5 &  12 &  12 &  12 & 8 & 25  & 11 & 7   \\
6 &  13 &  13 &  13 & 8 & 26  & 11 & 7   \\
7 &  14 &  14 &  14 & 8 & 25  & 11 & 7   \\
8 &  14 &  14 &  14 & 8 & 25  & 11 & 7   \\
9 &  14 &  14 &  14 & 8 & 25  & 11 & 7   \\
\end{tabular}
\label{tab:G5a}
\end{table}

It can be observed, that the iteration count flattens as in the other cases, the performance of the method is unaffected by the increase in the amount of groups or their different scaling.

We also show the results for the use of more than 1 pre and post smoothing in Table \ref{tab:G5s}.
\begin{table}
\centering
\caption[]{GMRES Iterations to reduce the residual by $10^{-8}$ for a 5 groups calculation, with different amount of smoothings per level, where "max" is the maximum amount of iterations for different $\varepsilon$.}
\begin{tabular}{c|ccc|ccc|ccc|}
& \multicolumn{3}{c|}{$2$ smoothings} & \multicolumn{3}{c|}{$4$ smoothings} & \multicolumn{3}{c|}{$8$ smoothings} \\
\diagbox{levels}{$\varepsilon$} & 1.0 &  0.1 &  0.01 & 1.0 &  0.1 &  0.01 & 1.0 &  0.1 &  0.01 \\
\hline
2 &  4 &  3 &  3  &  3  &  2  &  2  &  2  &  2  &  2  \\
3 &  5 &  5 &  5  &  4  &  4  &  4  &  3  &  3  &  3  \\
4 &  7 &  7 &  7  &  5  &  5  &  5  &  4  &  4  &  4  \\
5 &  8 &  8 &  8  &  6  &  6  &  6  &  5  &  5  &  5  \\
6 &  9 &  9 &  9  &  7  &  7  &  7  &  6  &  6  &  6  \\
7 &  9 &  9 &  9  &  7  &  7  &  7  &  7  &  7  &  7  \\
8 &  9 &  9 &  9  &  7  &  7  &  7  &  6  &  7  &  7  \\
9 &  9 &  9 &  9  &  6  &  7  &  7  &  6  &  7  &  7      
\end{tabular}
\label{tab:G5s}
\end{table}

We see an improvement in the iteration count, always flattening, that becomes less significant as the amount of smoothing iterations increases, suggesting that there is a sweet spot to be found with regards to the computational cost.

\subsection{Space dependent reaction matrix}
We modify the matrix used in the previous section scaling it with the following function only depending on space:
\begin{align*}
f_i(x,y) = 
\begin{cases}
(x,y) \in \boldsymbol\Omega_i & \sin^2(2\pi x)\sin^2(2\pi y) \\
(x,y) \notin \boldsymbol\Omega_i & 0
\end{cases}
\end{align*}
where $\boldsymbol\Omega_i$, with $i=0,1,2,3$ are the four quadrants of the square domain. Note that these results in reaction and diffusion dominated regions and inertial subspaces in group space depending on the spatial coordinates.
\begin{align*}
\boldsymbol{\Sigma} = \begin{pmatrix} 
\displaystyle \changed{\Sigma_{1,1}} & -f_0 & -\varepsilon^{-1}f_1 & -\varepsilon^{-2}f_2 & -\varepsilon^{-3}f_3 & \dots\\ 
-f_0 & \changed{\Sigma_{2,2}} & -f_0 & -f_0 & -f_0 & \dots\\
-\varepsilon^{-1}f_1 & -f_0 & \changed{\Sigma_{3,3}} & -\varepsilon^{-1}f_1 & -\varepsilon^{-2}f_2 & \dots \\
-\varepsilon^{-2}f_2 & -f_0 & -\varepsilon^{-1}f_1 & \changed{\Sigma_{4,4}} & -\varepsilon^{-1}f_1 & \dots \\
-\varepsilon^{-3}f_3 & -f_0 & -\varepsilon^{-2}f_2 & -\varepsilon^{-1}f_1 & \changed{\Sigma_{5,5}} & \dots \\
\vdots & \vdots & \vdots & \vdots & \vdots & \ddots
\end{pmatrix}
\end{align*}
Results are shown in Table \ref{tab:G5b} for different source terms as in the previous section. In this case the columns are shown only up to $\varepsilon=0.01$ to avoid a floating point underflow.

\begin{table}[ht!]
\centering
\caption[]{GMRES Iterations to reduce the residual by $10^{-8}$ for a 5 groups calculation, where "max" is the maximum amount of iterations for different $\varepsilon$.}
\begin{tabular}{c|ccc|c|c|c|c}
& \multicolumn{3}{c|}{MGAS} & MGMS & 2AS & 2HS & 2MS\\
\diagbox{levels}{$\varepsilon$} & 1.0 &  0.1 &  0.01 & max & max & max & max \\
\hline
2 &  6 &  7 &  6  & 4 & 19 & 7 & 4 \\
3 &  9 &  10 &  9  & 6 & 22 & 10 & 6 \\
4 &  11 &  12 &  12  & 7 & 25 & 11 & 7 \\
5 &  13 &  13 &  13  & 8 & 27 & 12 & 8 \\
6 &  13 &  14 &  14  & 8 & 28 & 12 & 8 \\
7 &  14 &  14 &  15  & 8 & 28 & 13 & 8 \\
8 &  14 &  15 &  15  & 9 & 27 & 12 & 8 \\
9 &  14 &  15 &  15  & 9 & 27 & 12 & 8 \\
10 &  15 &  15 &  15  & 9 & 27 & 12 & 8 \\
11 &  15 &  15 &  15  & 9 & 27 & 12 & 8 \\
12 &  15 &  15 &  15  & 9 & 27 & 12 & 8
\end{tabular}
\label{tab:G5b}
\end{table}

We see that once again, we achieve a flat iteration count, with a slightly larger absolute value for the finest meshes. The reaction term does not affect the convergence of the method, even when the reaction coefficients vary in space, as well as between groups.

\section{Conclusions}
We introduced a domain decomposition smoother based on the solution of the complete reaction-diffusion system on each cell of the mesh in the fashion of additive or multiplicative nonoverlapping Schwarz methods. We prove that these smoothers produce two-level and multilevel preconditioners performing robustly with respect to mesh size and parameters of the equation. Our numerical experiments confirm the robustness and show, that the obtained iteration counts are indeed low, and thus the methods very efficient.

\FloatBarrier
\bibliography{symmetric_diffusion_paper}{}
\bibliographystyle{siam}
\end{document}